\documentclass[12pt]{amsart}

\usepackage[margin=3cm]{geometry}

\usepackage{amsmath,amssymb,amsbsy,amsfonts,amsthm,latexsym,
                        amsopn,amstext,amsxtra,euscript,amscd,mathrsfs,color,bm, cite}
                        \usepackage{ulem}

\numberwithin{equation}{section}

\newtheorem{cor}{Corollary}
\newtheorem{theorem}{Theorem}

\begin{document}


\def\cA{{\mathcal A}}
\def\cB{{\mathcal B}}
\def\cC{{\mathcal C}}
\def\cD{{\mathcal D}}
\def\cE{{\mathcal E}}
\def\cF{{\mathcal F}}
\def\cG{{\mathcal G}}
\def\cH{{\mathcal H}}
\def\cI{{\mathcal I}}
\def\cJ{{\mathcal J}}
\def\cK{{\mathcal K}}
\def\cL{{\mathcal L}}
\def\cM{{\mathcal M}}
\def\cN{{\mathcal N}}
\def\cO{{\mathcal O}}
\def\cP{{\mathcal P}}
\def\cQ{{\mathcal Q}}
\def\cR{{\mathcal R}}
\def\cS{{\mathcal S}}
\def\cT{{\mathcal T}}
\def\cU{{\mathcal U}}
\def\cV{{\mathcal V}}
\def\cW{{\mathcal W}}
\def\cX{{\mathcal X}}
\def\cY{{\mathcal Y}}
\def\cZ{{\mathcal Z}}

\def\A{\mathbb{A}}
\def\B{\mathbf{B}}
\def \C{\mathbb{C}}
\def \F{\mathbb{F}}
\def \K{\mathbb{K}}

\def \Z{\mathbb{Z}}
\def \P{\mathbb{P}}
\def \R{\mathbb{R}}
\def \Q{\mathbb{Q}}
\def \N{\mathbb{N}}
\def \Z{\mathbb{Z}}

\title{On Sidon sets with squares, cubes and quartics in short intervals}

\author[M.~Z.~Garaev]{M.~Z.~Garaev}
\address{Centro  de Ciencias Matem{\'a}ticas,  Universidad Nacional Aut\'onoma de
M{\'e}\-xico, C.P. 58089, Morelia, Michoac{\'a}n, M{\'e}xico}
\email{garaev@matmor.unam.mx}

\author[F. M. Garayev] {F. M. Garayev}
\address{Korea Advanced Institute of Science and Technology, Daejeon 34141, Republic of Korea}
\email{fgarayev@kaist.ac.kr}

\author[S. V. Konyagin] {S. V. Konyagin}
\address{Steklov Mathematical Institute of Russian Academy of Sciences, 8 Gubkina St., 119991 Moscow, Russia;
Dept. Mech. and Math., Lomonosov Moscow State University, 1, Leninskie Gory, GSP 1, 119991 Moscow, Russia}
\email{konyagin@mi-ras.ru}

\begin{abstract} Representative examples of our results are as follows.

For any positive integer $N$ the equation
\begin{equation}
\label{eqn:abstract}
x^3+y^3=z^3+t^3, \quad x,y,z,t\in \mathbb{N}, \quad \{x,y\}\not=\{z,t\}
\end{equation}
has no solutions satisfying
$$
N\le x,y,z,t < N+\Bigl(\frac{38}{3}N+\frac{1297}{36}\Bigr)^{1/2}+\frac{19}{6}.
$$
The strict inequality ``$<$" can not be substituted by ``$\le$", that is, there exist infinitely many positive integers $N$ such that the equation~\eqref{eqn:abstract} has a solution with
$$
N\le x,y,z,t \le N+\Bigl(\frac{38}{3}N+\frac{1297}{36}\Bigr)^{1/2}+\frac{19}{6}.
$$

There is an absolute constant $c>0$ such that for any positive integer $N$ the equation~\eqref{eqn:abstract} has a solution  satisfying
$$
N\le x,y,z,t \le N+cN^{2/3}.
$$

For any $\varepsilon>0$ there exist infinitely many positive integers $N$ such that the equation~\eqref{eqn:abstract} has no solutions satisfying
$$
N\le x,y,z,t \le N+N^{4/7-\varepsilon}.
$$

There is an absolute constant $c>0$ such that for any positive integer $N$ the equation
\begin{equation}
\label{eqn:quartic intro}
x^4+y^4=z^4+t^4,\quad x,y,z,t\in\mathbb{N}, \quad \{x,y\}\not=\{z,t\},
\end{equation}
has no solutions satisfying
$$
N\le x,y,z,t \le N+cN^{3/5}.
$$

There is an absolute constant $c>0$ such that for any positive integer $N$ the equation~\eqref{eqn:quartic intro}
has a solution  satisfying
$$
N\le x,y,z,t \le N+cN^{12/13}.
$$

\end{abstract}


\maketitle


\section{\bf Introduction}

\bigskip

It has been known for a long time that the diophantine equation
$$
x^3+y^3=z^3+t^3,
$$
as well as the equation
$$
x^4+y^4=z^4+t^4,
$$
have non-trivial solutions in positive integers $x,y,z,t,$ see~\cite[Chapter XXI]{LED}, where numerous classical results have been cited with references. This means that the set of cubes $\{n^3: n \in\N\}$ as well as the set of fourth powers $\{n^4: n\in\N\}$ are not Sidon sets. We recall that a set $\cA$ is called a Sidon set if the equation
$$
a_1+a_2=a_3+a_4,\quad a_i\in\cA, \, i=1,2,3,4,
$$
has only trivial solutions, i.e., if $\{a_1,a_2\}=\{a_3,a_4\}.$ It is a widely open problem whether there exists a polynomial with integer coefficients whose values at positive integers form a Sidon set. A conjecture of Erd\H{o}s claims that the set of fifth powers $\{n^5: n\in\N\}$
should be a Sidon set. As an approximation to these problems, Ruzsa~\cite{Ruz} proved that there is a real number $\xi\in [0, 1]$ such that the set
$A = \{n^5 + [\xi n^4] : n>n_0\}$ is a Sidon set for a suitable constant $n_0.$

We also refer the reader to the work of Dubickas and Novikas~\cite{DuNo}, where it is shown that there does not exist a cubic polynomial with integer coefficients whose values at positive integers form a Sidon set.

\smallskip

When $n$ runs over integers from a short interval, the sets of squares $n^2$ and cubes $n^3$ recently appeared  in the work of Gabdullin~\cite{Gab} and, Gabdullin and Konyagin~\cite{GK} respectively.

The main focus of the work of Gabdullin~\cite{Gab} was  to investigate trigonometric polynomials with frequencies in the set of squares.
Cilleruelo and C\'ordoba~\cite{CC} posed the following conjecture about an estimate of $L^4$-norm of such trigonometric
polynomials in terms of $L^2$-norm: for any $\gamma\in (0,1)$ there is a number $C(\gamma)$ such that for any positive integer
$N$ and any trigonometric polynomial $t$ with frequencies from the set $\{n^2: N\le n\le N+N^{\gamma}\},$ the inequality
$\|t\|_4 \le C(\gamma) \|t\|_2$ holds. (See the comments in~\cite{Gab}.) Gabdullin proved this conjecture for
$\gamma < (\sqrt 5 - 1)/2$. The best estimates of $L^4$-norm of trigonometric polynomials with frequencies in $B\subset\mathbb Z$
in terms of $L^2$-norm can be written if $B$ is a Sidon set: in this case we have $\|t\|_4 \le 2^{1/4} \|t\|_2$.
Gabdullin~\cite{Gab} also proved that the set $\{n^2: \, N\le n\le N+(8N)^{1/2}\}$ is a Sidon set and observed that the constant
$8$ is sharp in the sense that for infinitely many $N$ the equation $x^2+y^2=z^2+t^2$ has a non-trivial solution in the interval
$[N, N+(8N)^{1/2}+3].$

\smallskip

Similarly, the main focus of the work of Gabdullin and Konyagin~\cite{GK} was to investigate trigonometric polynomials
with frequencies in the set of cubes
and to estimate $L^4$-norm of a trigonometric polynomial with frequencies in  
$\{n^3: N\le n\le N+N^{\gamma}\},$
$0 < \gamma < 2/3$, in terms of its $L^2$-norm.
They also proved the following statement.

\smallskip

{\bf Theorem} (Gabdullin and Konyagin). {\it The set $\{n^3: \, N\le n\le N+(0.5N)^{1/2}\}$
is a Sidon set. Moreover, this statement is sharp up to the constant $(0.5)^{1/2}.$}

\smallskip

In the present paper we refine these results of~\cite{Gab,GK} on Sidon sets and develop the topic in several directions.

\section{\bf{Our results}}

In this section we state our results. We start with the case of squares.

\begin{theorem}
\label{thm:equal sum squares}
For any positive integer $N$ the set
$$
\Bigl\{n^2: N \le n < N+ (8N+8)^{1/2}+2\Bigr\}
$$
is a Sidon set. Moreover, the strict inequality ``$<$" can not be substituted by ``$\le$", that is, there exist infinitely many positive integers $N$ such that the set
$$
\Bigl\{n^2: N \le n \le N+ (8N+8)^{1/2}+2\Bigr\}
$$
is not a Sidon set.
\end{theorem}

\begin{cor}
\label{cor:equal sum squares}
For any positive integer $N$ the set
$$
\Bigl\{n^2: N \le n < N+ (8N)^{1/2}+2\Bigr\}
$$
is a Sidon set. Moreover, the constants $8^{1/2}$ and $2$ are sharp in the sense that for any $\varepsilon>0$ there exist infinitely many positive integers $N$ such that the set
$$
\Bigl\{n^2: N \le n < N+ (8N)^{1/2}+2+\varepsilon\Bigr\}
$$
is not a Sidon set.
\end{cor}

\begin{theorem}
\label{thm:equal sum squares strong N}
For any fixed $\varepsilon>0$ and any sufficiently large  $N>N_0(\varepsilon),$ the set
$$
\Bigl\{n^2: N \le n\le N+\sqrt{(8+\varepsilon)N}\Bigr\}
$$
is not a Sidon set.
\end{theorem}

We now turn to the case of cubes.

\begin{theorem}
\label{thm:38/3}
For any positive integer $N,$ the set
$$
\Bigl\{n^3:\, N\le n < N+\Bigl(\frac{38}{3}N+\frac{1297}{36}\Bigr)^{1/2}+\frac{19}{6}\Bigr\}.
$$
is a Sidon set. Moreover, the strict inequality ``$<$" can not be substituted by ``$\le$", that is, there exist infinitely many positive integers $N$ such that the set
$$
\Bigl\{n^3:\, N\le n \le N+\Bigl(\frac{38}{3}N+\frac{1297}{36}\Bigr)^{1/2}+\frac{19}{6}\Bigr\}.
$$
is not a Sidon set.
\end{theorem}

\begin{cor}
\label{cor:38/3}
The set
$$
\Bigl\{n^3: N\le n \le N+\Bigl(\frac{38}{3}N\Bigr)^{1/2}+\frac{19}{6}\Bigr\}
$$
is a Sidon set. Moreover, the constants $38/3$ and $19/6$ are sharp in the sense that for any $\varepsilon>0$ there exists infinitely many positive integers $N$ such that the set
$$
\Bigl\{n^3: N\le n \le N+\Bigl(\frac{38}{3}N\Bigr)^{1/2}+\frac{19}{6}+\varepsilon\Bigr\}
$$
is not a Sidon set.
\end{cor}

\begin{theorem}
\label{thm:N^{2/3}}
There is an absolute constant $c>0$ such that for any positive integer $N$ the set
$$
\{n^3: N\le n \le N+cN^{2/3}\}
$$
is not a Sidon set.
\end{theorem}

In other words for any positive integer $N,$ the equation
$$
x^3+y^3=z^3+t^3, \quad x,y,z,t\in \mathbb{N}
$$
has a non-trivial solution  satisfying
$$
N\le x,y,z,t \le N+cN^{2/3}.
$$

Our next result shows that the analogy of Theorem~\ref{thm:equal sum squares strong N} does not hold for cubes.

\begin{theorem}
\label{thm:N^{4/7-}}
For any $\varepsilon>0$ there exists infinitely many positive integers $N$ such that the set
$$
\{n^3: N\le n \le N+N^{4/7-\varepsilon}\}
$$
is a Sidon set.
\end{theorem}

Finally,  we consider the case of fourth powers. It is not difficult to prove that for some absolute constant $c>0$ the set
$$
\Bigl\{n^4: N\le n\le N+cN^{1/2}\Bigr\},
$$
is a Sidon set. The main problem we consider here is to replace $N^{1/2}$ by $N^{1/2+c_0}$ with some $c_0>0.$ Next result shows that indeed $N^{1/2}$ can be substituted by $N^{3/5}.$

\begin{theorem}
\label{thm:quartic is Sidon} There is an absolute constant $c>0$ such that for any positive integer~$N$ the set
$$
\Bigl\{n^4: N\le n\le N+cN^{3/5}\Bigr\}
$$
is a Sidon set.
\end{theorem}

We did not make any attempts towards improving the exponent $3/5.$

On the opposite direction we have the following result.

\begin{theorem}
\label{thm:fourth power}
There is an absolute constant $c>0$ such that for any positive integer~$N$ the set
$$
\{n^4: N\le n \le N+cN^{12/13}\}
$$
is not a Sidon set.
\end{theorem}

In other words, for any positive integer $N$, the equation
$$
x^4+y^4=z^4+t^4, \quad x,y,z,t\in \mathbb{N}
$$
has a non-trivial solution  satisfying
$$
N\le x,y,z,t \le N+cN^{12/13}.
$$

\section{\bf {Proof of Theorem~\ref{thm:equal sum squares}}}

The fact that there exist infinitely many positive integers $N$ such that the set
$$
\Bigl\{n^2: N \le n \le N+ (8N+8)^{1/2}+2\Bigr\}
$$
is not a Sidon set immediately follows from the example given in~\cite{Gab} and credited to Alexander Kalmynin.

Let us prove that for any positive integer $N$ the set
$$
\Bigl\{n^2: N \le n < N+ (8N+8)^{1/2}+2\Bigr\}
$$
is a Sidon set.

Let $x^2 + y^2 = z^2 + t^2$.
If $x+y=z+t$ then we get the requested claim. Therefore, we can assume that $x+y\not=z+t.$ Without loss of generality we can further assume that $x+y>z+t, \, x\ge y, \, z\ge t.$
Let
$$
x=N+s_1, \quad y=N+s_2, \quad z=N+s_3, \quad t=N+s_4,
$$
where
$$
0\le s_1, s_2, s_3, s_4<\sqrt{8N+8}+2,\quad s_1\ge s_2, \quad s_3\ge s_4, \quad s_1+s_2>s_3+s_4
$$
and
\begin{equation}
\label{eqn:precise square 2N s1+s2+s3+s4=}
2N(s_1+s_2-s_3-s_4) = s_3^2+s_4^2-s_1^2-s_2^2.
\end{equation}
Clearly,  $s_1+s_2$ and $s_3+s_4$ of the same parity, so that we have
$$
s_1+s_2-s_3-s_4\ge 2.
$$
Taking this into account in the~\eqref{eqn:precise square 2N s1+s2+s3+s4=} and noting that
$$
|s_3-s_4-2|<\sqrt{8N+8},
$$
we obtain

\begin{eqnarray*}
&&4N\le 2N(s_1+s_2-s_3-s_4)\\ &&= s_3^2+s_4^2-s_1^2-s_2^2\le s_3^2 +s_4^2-\frac{(s_1+s_2)^2}{2}\\
&&\le s_3^2 +(s_4+2)^2-4s_4-4-\frac{(s_3+s_4+2)^2}{2}\\
&& = \frac{(s_3-s_4-2)^2}{2}-4s_4-4<\frac{8N+8}{2}-4 =4N,
\end{eqnarray*}
i.e. $4N<4N.$ This contradiction proves the first part of the theorem.

\section{\bf {Proof of Theorem~\ref{thm:equal sum squares strong N}}}

We shall prove that for any sufficiently large $N>N_0(\varepsilon)$ the equation
$$
x^2+y^2=z^2+t^2, \quad x,y,z,t\in \mathbb{N},
$$
has a non-trivial solution satisfying
$$
N \le x,y,z,t\le N+ \sqrt{(8+\varepsilon)N}.
$$

We assume that $N$ is a large positive integer. We take
$$
u = \lceil \sqrt {2N + 3} \rceil + 1.
$$
where $\lceil \sqrt {2N + 3}\rceil$ is the smallest integer greater than or equal to $\sqrt{2N+3}.$ Then
\begin{equation}
\label{eqn:square def u}
(u-2)^2 < 2N + 3 \le (u-1)^2.
\end{equation}

By the right hand side of inequality~\eqref{eqn:square def u}, the inequality
\begin{equation}
\label{eqn:square holds s=0, 1}
\frac {1 + u^2 - s^2}2 \ge N + u + 1
\end{equation}
holds for $s=0$ and $s=1$. Now let $s$ be
the largest integer such that $u - s \equiv 1(\bmod 2)$
and~\eqref{eqn:square holds s=0, 1} holds. By the left hand side of inequality~\eqref{eqn:square def u},
\begin{equation}
\label{eqn:square s2 le 2u-2}
s^2 \le 2u - 2.
\end{equation}

Denote
\begin{equation}
\label{eqn:square def l}
l = \frac {1 + u^2 - s^2}2.
\end{equation}
By~\eqref{eqn:square holds s=0, 1},
$$ l \ge N + u + 1.$$
By~\eqref{eqn:square def l}, we also have that
$$ (l - u -1)^2 + (l + u -1)^2 = (l-s)^2 + (l + s)^2.$$
In particular, the set
$$ \{n^2: N \le n \le l + u  -1\}$$
is not a Sidon set.
On the other hand, due to the choice of $s$,
$$ \frac {1 + u^2 - (s+2)^2}2 < N + u + 1.$$
Hence,
$$ l < N + u + 1 + 2s + 2.$$
Taking into account~\eqref{eqn:square holds s=0, 1} and~\eqref{eqn:square s2 le 2u-2}, we get
$$ 0 \le l - (N + u + 1) <4 N^{1/4}.$$
Therefore,
$$ (l + u - 1) - (N + 2u) <4 N^{1/4}.$$
Since $2u=2\lceil \sqrt {2N + 3} \rceil + 2 <(8N)^{1/2}+4,$ we conclude that
$$
l + u - 1 <  N + (8N)^{1/2} + 5 N^{1/4},
$$
finishing the proof of Theorem~\ref{thm:equal sum squares}.

\section{\bf Proof of Theorem~\ref{thm:38/3}}

In~\cite{GK}  the solvability of the equation
\begin{equation}
\label{eqn: 4cubesInTheorem}
x^3+y^3=z^3+t^3, \quad x,y,z,t\in \mathbb{N}, \quad \{x,y\}\not=\{z,t\}
\end{equation}
in short intervals was derived from the solvability of the generalized Pell equation
$$
X^2- 7Y^2 = 114.
$$
To prove Theorem~\ref{thm:38/3}  we follow the same line, however, we will arrive at the  generalized Pell equation
\begin{equation}
\label{eqn:aX^2-(a+18)Y^2}
aX^2-(a+18)Y^2 = 2a^2+36a+216,
\end{equation}
with $a\in\N,$ and  $X$ and $Y$ are integers of the same parity. In particular, we  use that for $a=1$ the equation~\eqref{eqn:aX^2-(a+18)Y^2} has infinitely many such solutions (note that for $a=3$ the equation~\eqref{eqn:aX^2-(a+18)Y^2} coincides with the above mentioned equation from~\cite{GK}).

\bigskip

\subsection{The first part of the theorem}

We will prove that the  equation~\eqref{eqn: 4cubesInTheorem} has no solutions subject to
\begin{equation}
\label{eqn:condition short interval}
N\le x,y,z,t < N+\Bigl(\frac{38}{3}N+\frac{1297}{36}\Bigr)^{1/2}+\frac{19}{6}.
\end{equation}

If $x+y=z+t,$ then we easily get that $\{x,y\}=\{z,t\}.$ So, we can assume that $x+y\not = z+t.$

Without loss of generality, we can also assume that
$$
x+y > z+t,\quad x\ge y,\quad z\ge t.
$$
Note that
$$
x+y\equiv x^3+y^3=z^3+t^3\equiv z+t\pmod 6.
$$
Therefore, $x+y=z+t+6k$ for some positive integer $k.$

\bigskip

{\bf Case 1.} Let $k\ge 2.$

Let us first show that $N\ge 14.$ Note that for any non-negative numbers $u,v$ we have
$$
u^3+v^3=\frac{(u+v)^3}{4}+\frac{3(u-v)^2(u+v)}{4},\qquad u^2+v^2\ge 2uv.
$$
We also observe that
$$
(u+v)^3= u^3+ 3uv^2 +(3u^2v+v^3)\ge u^3+(3+2\sqrt 3) v^2u>u^3+6.4uv^2.
$$
Hence,
\begin{eqnarray*}
z^3+t^3=x^3+y^3\ge \frac{(x+y)^3}{4}\ge \frac{(z+t+12)^3}{4} \\ \ge
\frac{(z+t)^3}{4}+\frac{6.4(z+t)\cdot 12^2}{4}.
\end{eqnarray*}
Then,
$$
\frac{3(z-t)^2(z+t)}{4}\ge \frac{6.4\cdot 12^2}{4}(z+t),
$$
whence $z-t>17.5$ Therefore,  the condition~\eqref{eqn:condition short interval} implies that
$$
\Bigl(\frac{38}{3}N+\frac{1297}{36}\Bigr)^{1/2}+\frac{19}{6}>z-t>17.5,
$$
from which we get $N>13.$
\bigskip

As in~\cite{GK}, we write
$$
x=N+s_1, \quad y= N+s_2, \quad z= N+s_3, \quad t = N+s_4,
$$
where the integers $s_i$ satisfy the conditions
$$
0\le s_i < \Bigl(\frac{38}{3}N+\frac{1297}{36}\Bigr)^{1/2}+\frac{19}{6}, \quad s_1+s_2-s_3-s_4\ge 12,\quad s_1\ge s_2,\quad s_3\ge s_4.
$$
Then,
$$
3(s_1+s_2-s_3-s_4)N^2 = 3N(s_3^2+s_4^2-s_1^2-s_2^2) + (s_3^3+s_4^3-s_1^3-s_2^3).
$$
It follows that
\begin{equation}
\label{eqn:s_1+s_2-s_3-s_4 }
36N^2 \le 3N(s_3^2+s_4^2-s_1^2-s_2^2) + (s_3^3+s_4^3-s_1^3-s_2^3).
\end{equation}
Let $0\le \delta\le 19/6$ be such that
$$
0\le s_3-s_4-\delta < \Bigl(\frac{38}{3}N+\frac{1297}{36}\Bigr)^{1/2}.
$$
Then
\begin{eqnarray*}
&&s_3^2+s_4^2-s_1^2-s_2^2\le s_3^2+(s_4+\delta)^2-\frac{(s_1+s_2)^2}{2} \\ && \quad \le s_3^2+(s_4+\delta)^2-\frac{(s_3+s_4+\delta + 8)^2}{2} \\ && \quad \le
s_3^2+(s_4+\delta)^2-\frac{(s_3+s_4+\delta)^2}{2}-8(s_3+s_4+\delta)
\\  && \quad = \frac{(s_3-s_4-\delta)^2}{2}-8(s_3+s_4+\delta)\\
&&\quad \le \frac{19N}{3}+\frac{1297}{72} - 8(s_3+s_4+\delta).
\end{eqnarray*}
Hence,
$$
s_3^2+s_4^2-s_1^2-s_2^2\le \frac{19N}{3} + \frac{55}{3} - 8(s_3+s_4+\delta).
$$
Inserting this into~\eqref{eqn:s_1+s_2-s_3-s_4 }, we get that

\begin{equation}
\label{eqn: s_3^2+s_4^2-s_1^2-s_2^2}
17N^2+24(s_3+s_4+\delta)N -55N \le s_3^3+s_4^3-s_1^3-s_2^3.
\end{equation}  
Similarly, taking into account the inequality
$$
(u + v)^3 = u^3 + 3uv^2 + (3u^2v + v^3) \ge u^3 + 3uv^2 + 2\sqrt{3} uv^2 \ge u^3 + 6.4 uv^2
$$
for $u \ge 0$, $v \ge 0$, we have
\begin{eqnarray*}
&&s_3^3+s_4^3-s_1^3-s_2^3\le s_3^3+s_4^3-\frac{(s_1+s_2)^3}{4} \\&& \quad \le  s_3^3+(s_4+\delta)^3-\frac{(s_3+s_4+\delta+8)^3}{4}\\ && \quad \le
s_3^3+(s_4+\delta)^3-\frac{(s_3+s_4+\delta)^3}{4}-\frac{6.4\cdot 8^2(s_3+s_4+\delta)}{4}\\ && \quad =  \frac{3(s_3-s_4-\delta)^2(s_3+s_4+\delta)}{4}-\frac{6.4\cdot 8^2(s_3+s_4+\delta)}{4}\\ && \quad \le \frac{(38N+110)(s_3+s_4+\delta)}{2}-\frac{6.4\cdot 8^2(s_3+s_4+\delta)}{4}\\&& \quad \le 19N(s_3+s_4+\delta).
\end{eqnarray*}
Since $N>13,$ we get contradiction with~\eqref{eqn: s_3^2+s_4^2-s_1^2-s_2^2}, which  finishes the case $k\ge 2.$

\bigskip

{\bf Case 2.} Let now $k=1.$

\bigskip

Denote
$$
A=x+y,\quad B=x-y, \quad C=z+t,\quad D= z-t.
$$
Then $A=C+6,$ and
\begin{equation}
\label{eqn:subst x=, y=, z=, t=}
x=\frac{C+6+B}{2}, \quad y = \frac{C+6-B}{2},\quad z=\frac{C+D}{2},\quad t=\frac{C-D}{2}.
\end{equation}
In particular, all the numbers $B, C, D$ are of the same parity. We also have that
$$
C-D=2t\ge 2N,\quad 0\le D =z-t < \Bigl(\frac{38}{3}N+\frac{1297}{36}\Bigr)^{1/2}+\frac{19}{6}.
$$

Inserting the substitution~\eqref{eqn:subst x=, y=, z=, t=} to our equation,  we get
$$
(C+6)^3+3(C+6)B^2 = C^3+3CD^2,
$$
i.e.
$$
6C^2+36C+72 +CB^2+6B^2 = CD^2.
$$
It follows that  $6B^2+72$ is divisible by $C.$
Let
$$
6B^2+72 = a_1C,\quad a_1\in\N.
$$
Then,
$$
6C+36+a_1+B^2 = D^2.
$$
Since $B$ and $D$ are of the same parity,  $a_1$ is an even number.
Letting $a_1=2a,$ we get that
\begin{equation}
\label{eqn:3B^2+36=aC}
3B^2+36 = aC,\quad a\in\N.
\end{equation}
and
\begin{equation}
\label{eqn:6C+36+2a+BB = DD}
6C+36+2a+B^2 = D^2.
\end{equation}
Multiplying by $a$ and using~\eqref{eqn:3B^2+36=aC}, we obtain that
$$
6(3B^2+36)+36a+2a^2+aB^2 = aD^2,
$$
whence
\begin{equation}
\label{eqn:aD^2-(a+18)B^2}
aD^2-(a+18)B^2 = 2a^2+36a+216.
\end{equation}

In particular,
$$
D \ge \Bigl(2a+36+\frac{216}{a}\Bigr)^{1/2}\ge (36+2\sqrt{432})^{1/2}>8,
$$
so that $D\ge 9.$
From~\eqref{eqn:3B^2+36=aC} and~\eqref{eqn:6C+36+2a+BB = DD} we also have that
$$
D^2=\frac{(a+18)C}{3}+2a+24.
$$
Hence, taking into account $2N\le 2t= C-D\le C-9,$ we get
\begin{eqnarray*}
&&\frac{(a+18)C}{3}+2a+24=D^2=(z-t)^2 \\ &&< \Bigl(\sqrt{\frac{38}{3}N+\frac{1297}{36}}+\frac{19}{6}\Bigr)^2 \le \Bigl(\sqrt{\frac{19(C-9)}{3}+\frac{1297}{36}}+\frac{19}{6}\Bigr)^2 \\&&=\frac{19(C-9)}{3}+\frac{1297}{36}+
\frac{19}{3}\sqrt{\frac{19(C-3)}{3}-38+\frac{1297}{36}} +\frac{361}{36}\\ &&
\le \frac{19C}{3}-57+\frac{1658}{36}+\frac{19}{3}\sqrt{\frac{19(C-3)}{3}}.
\end{eqnarray*}
Therefore,
$$
\frac{(a+18)C}{3}+2a+24\le \frac{19C}{3}-10+\frac{19}{3}\Bigl(\frac{19(C-3)}{3}\Bigr)^{1/2},
$$
implying
\begin{equation}
\label{eqn:2a+45 less than square root 19 cube C}
\frac{(a-1)C}{3}+2a+34\le \Bigl(\frac{19^3(C-3)}{27}\Bigr)^{1/2}.
\end{equation}
Furthermore,
$$
\frac{(a-1)C}{3}+2a+34=\frac{(a-1)(C-3)}{3}+3(a+11)\ge 2\Bigl((a-1)(a+11)(C-3)\Bigr)^{1/2}
$$
Inserting this into the preceding inequality, we get
$$
 2\Bigl((a-1)(a+11)\Bigr)^{1/2}\le \Bigl(\frac{19^3}{27}\Bigr)^{1/2}.
$$
whence $a < 6,$ i.e. $a\le 5.$

From~\eqref{eqn:2a+45 less than square root 19 cube C} we  have
$$
(a-1)^2C\le \frac{19^3}{3},
$$
and from \eqref{eqn:3B^2+36=aC} we have that
$aC >3B^2.$ Next,
\begin{equation}
\label{eqn:B vs 19 cube}
\frac{(a-1)^2}{a}B^2\le \frac{19^3}{9}.
\end{equation}

Considering the equation~\eqref{eqn:aD^2-(a+18)B^2}
modulo $5$ for $a\in\{2,5\},$ and modulo $11$ for $a=4$, we see that for $a\in\{2,4,5\}$ the equation~\eqref{eqn:aD^2-(a+18)B^2} has no solutions.

Thus, $a\in\{1,3\}$.

Let $a=3.$ From~\eqref{eqn:B vs 19 cube} we see that $B<24$ and the equation~\eqref{eqn:aD^2-(a+18)B^2} takes the form
$$
D^2-7B^2=114.
$$

The only solution of this equation with $B<24$ is $D=11, B=1$ giving $C=13,\, A=19$ and leading to the quadruple
$$
(x,y,z,t)=(10, 9, 12, 1).
$$
This however, does not satisfy the condition
$$
z<t+\sqrt{\frac{38t}{3}+\frac{1297}{36}}+\frac{19}{6}.
$$

Thus, we are only left with the case $a=1.$ It follows that
$$
D^2 = 19B^2+254,\quad C=3B^2+36.
$$
We have therefore,
\begin{eqnarray*}
\frac{19(C-D)}{3}+\frac{1297}{36}=\frac{19(3B^2+36)}{3}-\frac{19D}{3}+\frac{1297}{36}
\\=(D^2-254)+228-\frac{19D}{3}+\frac{1297}{36}=\Bigl(D-\frac{19}{6}\Bigl)^2.
\end{eqnarray*}
Since $C-D=2t\ge 2N,$ we get
\begin{eqnarray*}
D^2=(z-t)^2 <\Bigl(\sqrt{\frac{38N}{3}+\frac{1297}{36}}+\frac{19}{6}\Bigr)^2
\\ \le \Bigl(\sqrt{\frac{19(C-D)}{3}+\frac{1297}{36}}+\frac{19}{6}\Bigr)^2=D^2.
\end{eqnarray*}
Hence, we get the contradiction $D^2<D^2$
finishing the proof of the first part of Theorem~\ref{thm:38/3}.

\bigskip

\subsection{The second part of the theorem} In order to prove the second part of the theorem, we consider the equation~\eqref{eqn:aD^2-(a+18)B^2} for $a=1,$ that is, the equation
\begin{equation}
\label{eqn:D^2-19B^2=254}
D^2-19B^2=254.
\end{equation}
It has infinitely many solutions in positive integers $D$ and $B.$ Indeed, we have the initial solution $D_0=27, B_0=5$ and then an infinite series of solutions can be constructed using the solutions of the corresponding Pell equation. Thus, we define the positive integers $D$ and $B$ from
$$
D+B\sqrt{19} =(27+5\sqrt{19})\Bigl(u_1+v_1\sqrt{19}\Bigr)^m,
$$
where $m$ is any non-negative integer, $(u_1, v_1)$ is the fundamental solution (or any other solution) of the equation $u^2-19v^2=1,$ which in our case is $u_1=170, v_1=39.$ We have
$$
B\sqrt{19}<D<B\sqrt{19}+6,
$$
and $B\to\infty$ as $m\to\infty.$

Thus, following the notation of the case $k=1$ of the first part of the proof of our theorem, we consequently find
$$
C=3B^2+36, \quad A= 3B^2 +42
$$
and get that the quadruple $(x,y,z,t)$ defined by
$$
x=\frac{3B^2+B+42}{2},\quad y=\frac{3B^2-B+42}{2},\quad z= \frac{3B^2+36+D}{2}, \quad t= \frac{3B^2+36-D}{2}
$$
satisfies the equation $x^3+y^3=z^3+t^3.$  Take
\begin{equation}
\label{eqn:def of N in theorem 1}
N=\frac{3B^2+36-D}{2}.
\end{equation}
It follows that
$$
z > x > y > t =N
$$
and $N\to\infty$ as $B\to\infty.$ Noting that
$$
z= N + D =N+ \Bigl(\frac{38}{3}N+\frac{1297}{36}\Bigr)^{1/2}+\frac{19}{6},
$$
we finish the proof of the second part of Theorem~\ref{thm:38/3}.

\section{\bf {Proof of Theorem~\ref{thm:N^{2/3}}}}

We give two proofs of the theorem, which leads to different kinds of quadruples. The first proof is a consequence of the parametric formula of Euler and Binet for the solution of the equation $x^3+y^3=z^3+t^3.$ This leads to solutions with $\gcd(x,y,z,t)=1.$ In the second proof of the theorem we get quadruples with $\gcd(x,y,z,t)\gg N^{1/3}.$

According to the formula of Euler and Binet, for any rational numbers $p$ and $q$ the numbers
\begin{eqnarray*}
x=1-(p-3q)(p^2+3q^2), \quad y=-1+(p+3q)(p^2+3q^2),\\
z= p+3q-(p^2+3q^2)^2, \quad t=-(p-3q)+(p^2+3q^2)^2,
\end{eqnarray*}
satisfy our equation,  see~\cite[pages 157-158]{Dav} for the details. Letting $p=1$ and $q=1/n,$ we arrive at the proportional to the above quadruple solution
$$
x=n^3-n^2+3n, \quad y= n^3 + n^2+3n, \quad z= n^3-2n^2-3, \quad t=n^3+2n^2+3.
$$
The result now follows by taking $n=\lceil N^{1/3}\rceil+1$ or $n=\lceil N^{1/3}\rceil+2.$

\bigskip

Our second proof is as follows. From the proof of Theorem~\ref{thm:38/3}, we see that if $D_m, B_m$ and $N_m$ are such that
$$
D_m+B_m\sqrt{19} = (27+5\sqrt{19})(170+39\sqrt{19})^m, \quad N_m=\frac{3B_m^2+36-D_m}{2},
$$
then $D_m^2-19B_m^2=254$ and there exists a non-trivial solution to $x^3+y^3=z^3+t^3$ satisfying
$$
N_m\le x,y,z,t\le N_m+4\sqrt{N_m}.
$$
Clearly, $B_{m+1}\ll B_m$ and hence, $N_{m+1}\ll N_{m},$ the implied constant in $\ll$ is absolute. Therefore, for some absolute constant $c_0,$ there exists $M\in [N^{2/3}, c_0N^{2/3}]$ such that the equality
$$
x_1^3+y_1^3 = z_1^3+t_1^3
$$
holds for some positive integers $x_1,y_1,z_1,t_1$ with
$$
\quad M\le x_1,y_1,z_1,t_1 \le M+4\sqrt M,\quad \{x_1,y_1\}\not=\{z_1,t_1\}.
$$
Taking $k=\lceil N/M\rceil$ to be the smallest integer greater or equal to $N/M,$ we get
$$
(kx_1)^3+(ky_1)^3=(kz_1)^3+(kt_1)^3.
$$
It remains to note that
$$
N\le kx_1, \, ky_1, \, kz_1,\, kt_1\le \Bigl(\frac{N}{M}+1\Bigr)(M+4\sqrt M)\le N+c N^{2/3}
$$
for some absolute constant $c.$

\section{\bf{Proof of Theorem~\ref{thm:N^{4/7-}}}}

Let $M>M_0(\varepsilon)$ be a sufficiently large number and let $T<M^{1/3}.$
We will assume that for $N\in [M, 2M]$ the set of cubes $\{n^3: N \le n  \le N + \sqrt{MT}\}$ is not a Sidon set. Our aim is to prove that $T>M^{1/7-\varepsilon}.$

Observe that the interval $[M, 2M]$ contains at least $0.5\sqrt{M/T}$ disjoint subintervals of the form
$[N \le n  \le N + \sqrt{MT}].$ According to our assumption, in each of these subintervals the equation $x^3+y^3=z^3+t^3$ has at least one non-trivial solution.

We consider one such subinterval and assume that the positive integers $x,y,z,t$ are such that
\begin{equation}
\label{eqn:equal cubes 4/7}
x^3+y^3=z^3+t^3,\quad N\le x, y, z, t\le N+ \sqrt{MT},\quad \{x,y\}\not=\{z,t\}.
\end{equation}

We first claim that
$$
|x+y-z-t|\le T.
$$
Indeed, for any non-negative numbers $u,v$ we have
$$
4(u^3+v^3) = (u+v)^3 +3(u+v)(u-v)^2\ge (u+v)^3.
$$
Hence, for some $d_1\ge 0,$
$$
\Bigl(4(x^3+y^3)\Bigr)^{1/3}= x+y+d_1
$$
It then follows that
$$
4(x^3+y^3)\ge (x+y)^3+3(x+y)^2d_1
$$
whence
$$
3(x+y)^2d_1\le 3(x+y)(x-y)^2.
$$
Therefore,
$$
0\le d_1\le \frac{(x-y)^2}{x+y}\le \frac{TM}{2N}\le T.
$$
Analogously, for some $d_2\ge 0,$ we have
$$
\Bigl(4(z^3+t^3)\Bigr)^{1/3}= z+t+d_2
$$
and
$$
0\le d_2\le \frac{(z-t)^2}{z+t}\le T.
$$
Since $x^3+y^3=z^3+t^3,$ we get the bound
\begin{equation}
\label{eqn: x+y-z-t <T}
|x+y-z-t|=|d_1-d_2|\le T.
\end{equation}
We can further assume $x+y>z+t,\, x \ge y,\, z \ge t.$ As in the proof of Theorem~\ref{thm:38/3}, denote
$$
A=x+y,\quad B=x-y, \quad C=z+t,\quad D= z-t.
$$
Then $A=C+6k$ for some $k\in\N$ and
\begin{equation}
\label{eqn:substitution 4/7}
x=\frac{C+6k+B}{2}, \quad y = \frac{C+6k-B}{2},\quad z=\frac{C+D}{2},\quad t=\frac{C-D}{2}.
\end{equation}
Note that
$$
C\ge 2N,\quad 0\le B, D\le \sqrt{MT}.
$$
By\eqref{eqn: x+y-z-t <T} we also have $k\le T.$

Inserting~\eqref{eqn:substitution 4/7} into~\eqref{eqn:equal cubes 4/7}, we get
$$
6kC^2+(6k)^2C +72k^3+CB^2+6kB^2=CD^2.
$$
It follows that for some $a_1\in\N,$
\begin{equation}
\label{eqn:6kB^2+72k^3=a_1C}
6kB^2+72k^3=a_1C.
\end{equation}
Then
$$
6kC+36k^2+B^2+a_1=D^2.
$$
Multiplying by $a_1$ and taking into account~\eqref{eqn:6kB^2+72k^3=a_1C}, we obtain
\begin{equation}
\label{eqn:GenPell4/7}
a_1D^2-(a_1+36k^2)B^2=a_1^2+36k^2a_1+432k^3.
\end{equation}

Now, from~\eqref{eqn:6kB^2+72k^3=a_1C} it follows that
$$
a_1=\frac{6kB^2+72k^3}{C}\le \frac{6T^2N+72T^3}{N}\le 10T^2.
$$

Thus, $k\le T$ and  $a_1\le 10T^2.$ In particular, there are at most $10T^3$ possibilities for the number of pairs $(a_1,k).$
For each such pair, we have the generalized Pell equation~\eqref{eqn:GenPell4/7} in variables $D$ and $B.$ Recall that $0\le B, D \le \sqrt{MT}\le M.$ According to~\cite[Lemma 3.5]{VW} (see, also~\cite[Proposition 1]{CG}), for a given pair $(a_1,k),$ the number of possibilities for the pair $(D, B)$ is at most $M^{\varepsilon}.$ This implies that for the given $a_1$ and $k$ there are at most $M^{\varepsilon}$ possibilities for the quadruple  $(x,y,z,t).$

Therefore, overall we have at most $T^3M^{\varepsilon}$ possibilities for the number of non-trivial quadruples  $(x,y,z,t)$ falling into short intervals of the form $[N, N+\sqrt{MT}].$ But, as we noticed, there are at least $0.5\sqrt{M/T}$ disjoint short subintervals of the given form and, by our assumption, each of these intervals contain at least one such quadruple
$(x,y,z,t).$

Therefore,
$$
T^3M^{\varepsilon}\ge 0.5\sqrt{M/T},
$$
whence $T\ge M^{1/7-\varepsilon},$ finishing the proof of the theorem.

\section{\bf {Proof of Theorem~\ref{thm:quartic is Sidon}}}

Let positive integers $x,y,z,t$ be such that
$$
x^4+y^4=z^4+t^4,\qquad \{x,y\}\not=\{z,t\}
$$
and
$$
N\le x,y,z,t\le N+N^{3/5}.
$$
Assume that
$$
 x\ge y, \quad z\ge t,\quad x+y\ge z+t.
$$
Our aim is to prove that $z-t\ge cN^{3/5}$ for some absolute constant $c>0.$

We can assume that $N$ is sufficiently large.

Note that $x+y\not=z+t.$ Indeed otherwise, denoting $x+y=z+t=h$ and assuming
$x\ge h/2,\, z\ge h/2$ we get
$$
x^4+(h-x)^4=z^4+(h-x)^4.
$$
The function $f(X)=X^4+(h-X)^4$ is strictly increasing in the interval $X\in [h/2, h]$ as  $f'(X)=4X^3-4(h-X)^3>0$ for $X\in (h/2, h).$ Hence, we get that $x=z.$

Thus, we can assume that
$$
x+y>z+t, \quad x\ge y, \quad z\ge t.
$$
Since $x+y\equiv z+t\pmod 2,$ we also have $x+y = z+t+2a$ for some positive integer~$a.$
Let
$$
p=x+y,\quad q= x-y,\quad u=z+t,\quad v=z-t.
$$
Then
\begin{equation}
\label{eqn:v vs u in 8/3}
p = u+2a,\quad u\ge 2N,\quad 0\le v \le N^{3/5}
\end{equation}
and
$$
x=\frac{p+q}{2}=\frac{u+2a+q}{2},\quad y=\frac{p-q}{2},\quad z=\frac{u+v}{2},\quad t=\frac{u-v}{2}.
$$
Substituting this in our equation, we see that
\begin{equation}
\label{eqn:8/3 after substitution}
(u+2a)^4+6(u+2a)^2q^2+q^4= u^4+6u^2v^2+v^4.
\end{equation}
Expanding the expression and cancelling $u^4$, we get
$$
8au^3+24a^2u^2+32a^3u+16a^4+6(u^2+4au+4a^2)q^2+q^4= 6u^2v^2+v^4.
$$
It follows that $v^2>4au/3,$ as otherwise the quantity on the right hand side would be smaller than the one on the left hand side. Thus, for some positive integer $m$ we have that
$$
3v^2=4au+m.
$$
Then,
\begin{eqnarray*}
8au^3+24a^2u^2+32a^3u+16a^4+6(u^2+4au+4a^2)q^2+q^4 \\= 2u^2(4au+m)+\frac{(4au+m)^2}{9}.
\end{eqnarray*}
Cancelling the terms $8au^3$ and then multiplying  by $9$ we get
\begin{eqnarray*}
9\cdot 24a^2u^2+9\cdot 32a^3u+9\cdot 16a^4+9\cdot 6(u^2+4au+4a^2)q^2+9q^4\\ = 18mu^2+(16a^2u^2+8amu+m^2),
\end{eqnarray*}
whence
\begin{eqnarray*}
9\cdot 32a^3u+9\cdot 16a^4+54(u^2+4au+4a^2)q^2+9q^4\\=
(18m-200a^2)u^2+ 8amu +m^2.
\end{eqnarray*}
Clearly, $m>3q^2,$ as otherwise the expression on the right hand side would be smaller than the one on the left hand side.
Thus, let
$$
m=3q^2+k_1,\quad k_1\in\N
$$
Substituting in our equation and cancelling $54q^2u^2$ and $9q^4$ we get
\begin{eqnarray*}
&&9\cdot 32a^3u+9\cdot 16a^4+54(4au+4a^2)q^2\\ && \quad =
(18k_1-200a^2)u^2+8au(3q^2+k_1) +(6k_1q^2+k_1^2).
\end{eqnarray*}
It follows that $k_1$ is even, and since $3v^2=4au+3q^2+k_1$ we get that $k_1$ is actually divisible by $4.$ Thus, we take $k_1=4k$ and after cancelling by $8,$ we get
$$
36a^3u+18a^4+27(au+a^2)q^2=(9k-25a^2)u^2+au(3q^2+4k)+(3kq^2+2k^2).
$$
Now collecting all the terms with $q^2$ on the right, and the remaining terms on the left, we arrive at the equation
\begin{equation}
\label{eqn:terms with q2 right the rest left}
3q^2(8au+9a^2-k)= (9k-25a^2)u^2+4au(k-9a^2)+2k^2-18a^4.
\end{equation}
We also recall that
$$
3v^2=4au+3q^2+4k \gg au+q^2.
$$
In particular, we can assume that $a<u^{1/3},$ as otherwise there is nothing to prove. We also recall that $u\ge 2N$ is sufficiently large.

If $k>au,$ then $k>100a^2, \, 9k-25a^2>au,$ and we get
$$
3q^2(8au+9a^2-k) < 24q^2au, \quad (9k-25a^2)u^2+4au(k-9a^2)+2k^2-18a^4> au^3.
$$
Therefore, in this case from~\eqref{eqn:terms with q2 right the rest left} we get $q^2\gg u^2,$ which implies even a stronger bound than the claim of the theorem.

Otherwise, $k\le au$ and in~\eqref{eqn:terms with q2 right the rest left} we pass to the congruence and get that
$$
(9k-25a^2)u^2+4au(k-9a^2)+2k^2-18a^4 \equiv 0\pmod{8au+9a^2-k}.
$$
We multiply this congruence by $64a^2$ and substitute
$$
8au\equiv k-9a^2 \pmod {8au+9a^2-k}.
$$
Then,
\begin{eqnarray*}
(9k-25a^2)(k-9a^2)^2+32a^2(k-9a^2)^2+128a^2(k^2-9a^4)\\ \equiv 0\pmod {8au+9a^2-k}.
\end{eqnarray*}
Therefore,
\begin{equation}
\label{eqn:dioph eq quartics k and a}
(9k+7a^2)(k-9a^2)^2+128a^2(k^2-9a^4)\equiv 0\pmod {8au+9a^2-k}.
\end{equation}
Decomposing the left hand side into a product, we get
$$
9(k-a^2)(k^2-2ka^2+65a^4)\equiv 0\pmod {8au+9a^2-k}.
$$
From~\eqref{eqn:terms with q2 right the rest left} it follows that $k>a^2.$ Therefore, since $k^2-2ka^2+65a^4>0,$ we get that the left hand side of the congruence~\eqref{eqn:dioph eq quartics k and a} as a number is distinct from $0$ and positive. Hence,
$$
1000k^3 > 9(k-a^2)(k^2-2ka^2+65a^4) \ge 8au+9a^2-k >au.
$$

Thus, we have that
\begin{equation}
\label{eqn:ineq for k,a,u}
k^3\gg au,
\end{equation}
where $L\gg M$ means that $L\ge c_0 M$ for some absolute constant $c_0>0.$

If $k<10a^2,$ then $a^6\gg k^3\gg au$ and then $a\gg u^{1/5}$ and the result follows from
$$
v^2\gg au +q^2 \gg u^{6/5}.
$$

If $k\ge 10a^2,$ then recalling that $k\le au,$ from~\eqref{eqn:terms with q2 right the rest left} we get $q^2au\gg ku^2$ and thus
$$
q^2\gg \frac{ku}{a}.
$$
Using the last inequality and (\ref{eqn:ineq for k,a,u}) we get
$$
v^2\gg au+q^2\gg au+\frac{ku}{a}\gg au + \frac{u^{4/3}}{a^{2/3}}\gg u^{6/5}.
$$

\section{\bf {Proof of Theorem~\ref{thm:fourth power}}}

Our proof is based on Euler's parametric solution to the equation
\begin{equation}
\label{eqn:A4+B4=C4+D4 proof Theorem}
A^4+B^4=C^4+D^4.
\end{equation}
His solution is described in~\cite[p. 644]{LED}. Let $b^2\not=1.$
Define
$$
f=\frac{3b^2-1}{2}, \qquad g=\frac{3b^4-18b^2-1}{8(b^2-1)},\qquad
z=\frac{b^2(b^2-4)-2fg}{b^2+g^2}.
$$
Put
\begin{align*}
p=&a(b^2-1-z), \quad &q=ab(b^2-1+fz+gz^2), \\
r=&ab(1+z)(b^2-1-z), \quad &s=a(b^2-1+fz+gz^2).
\end{align*}
Euler proved that the following quadruple $(A,B,C,D)$ satisfy the equation~\eqref{eqn:A4+B4=C4+D4 proof Theorem}:
\begin{align*}
A= p+q,\quad
B= r-s,\quad
C= r+s ,\quad
D= p-q.
\end{align*}

In Euler's formula we take $b=1+\frac{2}{n}$ and eventually arrive at the following solution:

\begin{align*}
A= n^{13} + 14n^{12} + 92n^{11} + 370n^{10} + 1002n^9 + 1908n^8 + 2611n^7 + 2609n^6
+1944n^5 +\\+ 1101n^4 + 458n^3 + 115n^2 + 7n - 1,\\\\
B = n^{13} + 12n^{12} + 65n^{11} + 204n^{10} + 384n^9 + 360n^8 - 122n^7 - 867n^6  -1266n^5 - \quad\,\,\, \\-1023n^4 - 493n^3 - 135n^2 - 20n - 3,\\\\
C = n^{13} + 12n^{12} + 68n^{11} + 246n^{10} + 642n^9 + 1272n^8 + 1939n^7 + 2265n^6 + 1992n^5 +\\+ 1263n^4 + 530n^3 + 129n^2 + 19n + 3,\\\\
D = n^{13} + 14n^{12} + 89n^{11} + 346n^{10} + 924n^9 + 1776n^8 + 2470n^7 + 2423n^6 + 1578n^5 +\\+ 591n^4 + 59n^3 - 47n^2 - 20n - 1.
\end{align*}

\bigskip

The claim of the theorem now follows by taking  $n=\lceil N^{1/13}\rceil.$

\bigskip

{\bf Acknowledgement.}

The authors are thankful to Mikhail Gabdullin for useful discussions.

The work of the third author was supported by the Russian Science Foundation under grant
no. 22-11-00129, https://rscf.ru/project/ 22-11-00129, and performed at Lomonosov Moscow State University.

\end{document}